\documentclass[12pt]{amsart}
\usepackage{amsmath,amssymb, amscd}

\theoremstyle{plain}
\newtheorem{thm}{Theorem}
\newtheorem{theorem}[thm]{Theorem}

\newtheorem{lemma}[thm]{Lemma}

\newtheorem{proposition}[thm]{Proposition}
\theoremstyle{definition}

\newtheorem{definition}[thm]{Definition}

\numberwithin{equation}{section}
\newcommand{\0}{{\mathcal O}}

\newcommand{\sO}{{\mathcal O}}

\newcommand{\cit}{{\mathbb C}}

\newcommand{\pit}{{\mathbb P}}

\newcommand{\fg}{{\mathfrak g}}

\newcommand{\fm}{{\mathfrak m}}

\newcommand{\ft}{{\mathfrak t}}

\newcommand\Spec{\rm Spec}

\def\codim{\mathop{\rm codim}\nolimits}

\title{Symplectic resolutions for conical symplectic varieties}
\author{Michel Brion and Baohua Fu}
\begin{document}
\maketitle

\begin{abstract}
We introduce the notion of a conical symplectic variety,
and show that any symplectic resolution of such a variety
is isomorphic to the Springer resolution of a nilpotent
orbit closure in a semisimple Lie algebra, composed with a linear projection.
\end{abstract}


\section{Introduction}
\begin{definition}\label{def}
A {\em conical symplectic variety} is an affine variety $W \subset \cit^N$, smooth in codimension 1,
such that (i) $W$ is invariant under the dilation action of $\cit^*$ on $\cit^N$, (ii)
There exists a holomorphic symplectic form $\omega$ on $W_{reg}$ such that $\lambda^* \omega = \lambda \omega$ for all $\lambda \in \cit^*$, (iii) For any
resolution $\pi:Z \to W$, the 2-form $\pi^* \omega$ (defined on $\pi^{-1}(W_{reg}))$ extends to a regular 2-form $\Omega$ on  $Z$.
 A resolution $\pi: Z \to W$ is called {\em symplectic} if $\Omega$ is a holomorphic symplectic form on the whole of $Z$.
\end{definition}

For a conical symplectic variety $W \subset \cit^N$, 
its normalization $\tilde{W}$ becomes a symplectic variety in the sense of Beauville \cite{B}.
Typical examples of conical symplectic varieties are nilpotent orbit closures
$\bar{\0}$ in a semi-simple Lie algebra $\mathfrak{g}$ and their symplectic resolutions can be constructed as follows.  For a flag variety $G/P$, its cotangent bundle $T^*_{G/P}$ admits a natural Hamiltonian $G$-action, and the image of the moment map 
$T^*_{G/P} \to \mathfrak{g}^* \simeq \mathfrak{g}$ is a nilpotent orbit closure 
$\bar{\0}_P$ (the Richardson orbit of $P$). Hence we have a projective generically 
finite morphism $\mu: T^*_{G/P} \to \bar{\0}_P$, which is called the Springer map 
associated to $G/P$. If $\mu$ is birational, it becomes a symplectic resolution 
of $\bar{\0}_P$, which we call a Springer resolution. In \cite{F}, it is proven 
that symplectic resolutions of nilpotent orbit closures are exactly Springer 
resolutions.

%

Further examples of conical symplectic varieties can be constructed as follows: 
take a nilpotent orbit $\0 \subset \fg$ and a linear subspace $L \subset \fg$ 
such that the natural projection $\fg \to \fg/L$ maps $\bar{\0}$ birationally 
to a subvariety $W$. 
Then $W$ is a conical symplectic variety provided it is smooth in codimension 1. 
This happens for example if $L$ is a line, not contained in
the secant variety of $\bar{\0}$.
We suspect that this construction gives all conical symplectic varieties. 

The main result of this note classifies symplectic resolutions for conical 
symplectic varieties. It generalizes the main theorem of \cite{F}, 
with a simpler proof.

\begin{theorem} \label{main}
Let $\pi: Z \to W$ be a symplectic resolution for a conical symplectic variety. Then 

(i)  $Z \simeq T^*_{G/P}$  for some flag variety $G/P$.

(ii) The Springer map $\mu: T^*_{G/P} \to \bar{\0}_P$ associated to $G/P$ 
is birational.

(iii) There exists a linear subspace $L \subset \fg$ such that 
$W \subset \cit^N$ is the image of  $\bar{\0}_P \subset \fg$ under the 
projection $p: \fg \to \fg/L $. The induced map $\bar{\0}_P \to W$ is birational.

(iv) The map $\pi$ is isomorphic to the composition $p \circ \mu$.

(v) If $W$ is normal, then $W = \bar{\0}_P$.
\end{theorem}

Remark that our assumptions in Definition \ref{def} are optimal for Theorem \ref{main}.  
This theorem is inspired by the alternative proof of \cite{F} provided by Namikawa in 
Section 8 of \cite{N2}.  
As an application, we provide a simpler proof of the main theorem of \cite{N2}, Section 7. 
\vspace{0.3 cm}

{\em Acknowledgements:} 
We are grateful to Yoshinori Namikawa for pointing out an error in
a previous version, and to the referees for their helpful comments. 
Baohua Fu is supported by NSFC (11031008 and 11225106).

\section{Proof of Theorem \ref{main}}

For a smooth variety $Y$, we denote by $T_Y$ the tangent bundle, by $T^*_Y$ the cotangent bundle,
and by $K_Y$ the canonical bundle (the determinant of $T^*_Y$). 
A contact structure on $Y$  is a corank one subbundle $F \subset T_Y$ such that the bilinear form on $F$ with
values in the quotient line bundle $L = T_Y/F$ deduced from the Lie bracket on $T_Y$ is everywhere 
non-degenerate. This implies that $\dim Y = 2n-1$ is odd and $K_Y \simeq L^{-n}$.  We call $L$ the contact 
line bundle of the contact structure on $Y$. A typical example of contact manifold is the projectivised cotangent 
bundle $\pit(T^*_M) := (T^*_M \setminus (\text{zero section}))/\cit^*$ of a smooth variety $M$, where the contact 
line bundle is $\0_{\pit(T^*_M)}(1)$.

\begin{proposition} \label{contact}
Let $X \subset \pit^{N-1}$ be a closed singular subvariety 
and $f: Y \to X$ a resolution. Assume that $Y$ has a contact structure with the contact line bundle $f^* \0_X(1)$. 
Let $\tilde{X} \to X$ be the normalization map. Then $Y \simeq \pit(T^*_{G/P})$ for some flag variety $G/P$ and 
the induced map $Y \to \tilde{X}$ yields the Stein factorization of the projectivised Springer map associated to $G/P$.
\end{proposition}
\begin{proof}
As $X$ is singular, we get $b_2(Y) \geq 2$. Note that  $K_Y = f^* \0_X(-n)$ with $n = (\dim X+1)/2$, 
hence $K_Y$ is not nef. By \cite{KPSW}, we deduce that 
$Y \simeq \pit(T^*_M)$ for some smooth projective variety $M$ and 
$\0_{\pit(T^*_M)}(1) \simeq f^* \0_X(1)$. This implies that $\0_{\pit(T^*_M)}(1)$,  hence
the tangent bundle $T_M$  is globally generated. 
We deduce that $M$ is a homogeneous variety, hence by \cite{BR} 
it is isomorphic to $G/P \times A$,
where $G/P$ is a flag variety for some semisimple algebraic group $G$ 
and $A$ is an abelian variety. 

Let $\hat{X} \subset \cit^N$ be the affine cone of $X$. As 
$f^*\0_X(-1) = \0_{\pit(T^*_M)}(-1)$, the map $f$ pulls back the $\cit^*$-bundle $\hat{X} \setminus \{0\}  \to X$ to the $\cit^*$-bundle 
$T^*_M \setminus \text{(zero section)} \to \pit(T^*_M) $, which gives a birational map $\hat{f}: T^*_M \setminus \text{(zero section)} \to \hat{X} \setminus \{0\}$. Note that $T^*_M \simeq T^*_{G/P} \times \cit^g \times A$, where $g = \dim A$. As $\hat{X}$ is affine, the map $\hat{f}$
contracts $\{x\} \times A$ to one point for any $x \in (T^*_{G/P}  \setminus \text{(zero section)}) \times (\cit^g \setminus \{0\})$, hence $\hat{f}$ factors through $T^*_M \setminus \text{(zero section)} \to (T^*_{G/P}  \setminus \text{(zero section)}) \times (\cit^g \setminus \{0\})$. As $\hat{f}$ is birational, we get $g=0$ and $M \simeq G/P$. If $\dim G/P =1$, then 
$G/P \simeq \pit^1$ and $X$ is smooth, which contradicts our assumption. 
Thus, $\dim G/P \geq 2$.

Let $\mu: T^*_{G/P} \to \bar{\0}$ be the Springer map associated to $G/P$;
this is a projective, generically finite morphism. By the Stein factorization, 
it follows that the algebra $A := H^0(T^*_{G/P}, \0_{T^*_{G/P}})$ is finitely generated, 
and $\mu$ factors as a birational morphism 
$\tilde{\mu} : T^*_{G/P} \to V := \Spec(A)$ followed by a finite
 morphism $\eta: V \to \bar{\0}$. 
Note that $\tilde{\mu}$ and $\eta$ are both $G \times \cit^*$-equivariant;
also, $V^{\cit^*}$ is a single point, say $o$.

Let $\cit[\hat{X}]$ be the coordinate ring of the affine variety $\hat{X}$, then the birational morphism $\hat{f}: T^*_{G/P} \setminus \text{(zero section)} \to \hat{X} \setminus \{0\} \subset \hat{X}$ induces a homomorphism of $\cit$-algebras $\cit[\hat{X}] \to H^0(T^*_{G/P} \setminus \text{(zero section)}, \0) = A$, where the latter equality follows from $\dim (G/P) \geq 2$. This gives a morphism $V \to \hat{X}$ and the map $\tilde{\mu}$ (restricted to $T^*_{G/P} \setminus \text{(zero section)}$) yields the Stein factorization of $\hat{f}$. Hence we get that $V \setminus \{o\}$ is the normalization of 
$\hat{X} \setminus \{0\}$. By taking the projectivisation, we get our claim.
\end{proof}

We now classify conical symplectic varieties with only isolated singularities.
\begin{lemma}\label{iso}
Let $W$ be a conical symplectic variety with only isolated singularities. Then  $W = \bar{\0}_{min}$, where $\0_{min}$
is the minimal nilpotent orbit in a simple Lie algebra. 
\end{lemma}
\begin{proof}
As $W$ is invariant by the $\cit^*$-action, it has a unique singular point, which is $\{0\}$.  As $W \setminus \{0\} \to \pit W$ is a $\cit^*$-bundle, we deduce that $\pit W$ is smooth. Note that by 
Lemma 1.4 \cite{B1}, the symplectic form on $W \setminus \{0\}$ induces a contact structure on $\pit W$ with contact line bundle $\0_{\pit W}(1)$. As   $\0_{\pit W}(1)$ is very ample, we deduce that $\pit W \simeq \pit \0_{min}$ by Cor. 1.8 \cite{B1}.
This gives that $W = \bar{\0}_{min}$.
\end{proof}

Now let us prove Theorem \ref{main}.
First note that the $\cit^*$-action on $W$ lifts to $Z$ (see for example 
Prop. A.7 of \cite{N1}), which makes $\pi$ to be $\cit^*$-equivariant.
Let $\Omega$ be the symplectic form on $Z$ extending $\pi^* \omega$, then  we have $\lambda^* \Omega = \lambda \Omega$. We denote by $\pit Z$ the quotient
$(Z \setminus \pi^{-1}(0))/\cit^*$. Then we get a morphism $\bar{\pi}: \pit Z \to \pit W$. 
By Sections 3 and 4 of \cite{N2}, we have
\begin{lemma} \label{l.Nam}
$\pit Z$ is a smooth contact projective variety with the contact line bundle
$L:=\bar{\pi}^* \0_{\pit W}(1)$.
\end{lemma}

If $\pit W$ is smooth, then by Lemma \ref{iso}, $W = \bar{\0}_{min} \subset \mathfrak{g}$. As $W$ admits a symplectic resolution, this implies that $\mathfrak{g}$ is of type $A$ and $Z \simeq T^*_{\pit^n}$ (cf. \cite{F} or Proposition \ref{sym} below). 
Assume now that $\pit W$ is singular, then we can apply Lemma \ref{l.Nam} and Proposition \ref{contact} to conclude that $(\pit Z, L) \simeq (\pit (T^*_{G/P}), \0_{\pit(T^*_{G/P})}(1))$. Note that we may take $G = {\rm Aut}^0(G/P)$,
up to changing $G$ and $P$.  
 By the proof of Proposition \ref{contact}, 
we have the following diagram:
\[ \CD
Z @<<<  Z \setminus \pi^{-1}(0) @ >\simeq>> T^*_{G/P}\setminus \text{(zero section)}  @ >>> T^*_{G/P} \\ 
@V\pi VV  @VVV @VVV @VV\tilde{\mu}V \\
W @<<< W \setminus \{0\} @<\text{normalization}<< V\setminus \{o\} @>>> V\simeq \tilde{W}
\endCD \]

As $\pit W$ is not smooth, we may assume $G/P \neq \pit^n$. Now we can apply Lemma 5 \cite{N2} and the argument in 
{\em loc. cit.} (p. 183) to deduce that $Z \simeq T^*_{G/P}$ and this identifies $\tilde{\pi}: Z \to \tilde{W}$ with the map 
$\tilde{\mu}$ from the Stein factorization of the Springer map $\mu: T^*_{G/P} \to \bar{\0}$. 
As seen in the proof of Proposition  \ref{contact}, it follows that $\tilde{W}$ admits a $G \times \cit^*$-action 
such that $\tilde{\pi}$ is equivariant.


As $W \subset \cit^N$ and $\cit^*$ acts on $\cit^N$ by dilations, 
the coordinate ring $\cit[W]$ is a subalgebra of $\cit[\tilde{W}]$ 
generated by elements of degree 1. On the other hand, we have  
$\cit[\tilde{W}] = H^0(T^*_{G/P}, \0_{T^*_{G/P}})$. The space of degree 1 
elements in $H^0(T^*_{G/P}, \0_{T^*_{G/P}})$ is $H^0(G/P, T_{G/P})$, 
which is $\mathfrak{aut}(G/P) = \fg$. Thus, the algebra $\cit[W]$ 
is generated by a linear subspace of $\fg$, that we view as the orthogonal 
of a linear subspace $L \subset \fg$. So the normalization map 
$\tilde{W} \to W$ factors as the map 
$\tilde{W} \to \bar{\0} \subset \fg \cong \fg^*$ corresponding 
to the inclusion $\fg \subset H^0(T^*_{G/P}, \0_{T^*_{G/P}})$, 
followed by a map $p: \bar{\0} \to W$, the restriction of the quotient 
map $\fg \to \fg/L$. Thus $p$ is birational. If $W$ is normal, then  
$W = \tilde{W}$, hence we get $W = \bar{\0}$.
This finishes the proof of Theorem \ref{main}. \\

%
%
%

In the proof of Theorem \ref{main}, we used \cite{F} to deduce that among minimal nilpotent orbit closures, only that of type $A$ admits a symplectic resolution. One can in fact prove a stronger result.
\begin{proposition} \label{sym}
Let $W$ be a normal variety of dimension $2n \geq 4$ with an isolated singularity $0$ and a contracting $\cit^*$-action.  Assume that $W \setminus \{0\}$ admits a symplectic form $\omega$ satisfying $\lambda^* \omega = \lambda \omega, \forall \lambda \in \cit^*$.
If $W$ admits  a symplectic resolution $\pi: Z \to W$, then $Z \simeq T^*_{\pit^n}$ and $W \simeq \bar{\0}_{min} \subset \mathfrak{sl}_{n+1}$.
\end{proposition}
\begin{proof}
Recall that the action of $\cit^*$ on $W$ lifts to an action on $Z$ (see Prop. A.7 of \cite{N1}).
Let $Z^{\cit^*}$ be the subvariety of $Z$ consisting of $\cit^*$-fixed points, which is a disjoint union of smooth subvarieties 
since $Z$ is smooth. As $W^{\cit^*} = \{0\}$, we have that $Z^{\cit^*}$ is contained
in the fiber $\pi^{-1}(0)$, hence it is a union of projective manifolds.

By the Bialynicki-Birula decomposition (see \cite{BB} which applies in this non-projective setting, since $Z$ is proper over 
$W$ which is contracted to $0$ by the $ \cit^*$-action), there exists an irreducible component $M$ of $Z^{\cit^*}$ 
such that the set $U:=\{z \in Z| \lim_{\lambda\to 0} \lambda \cdot z \in M\}$
is open in $Z$. Let $\Omega$ be the symplectic form on $Z$, which is the extension of $\pi^* \omega$.
 As $\lambda^* \Omega = \lambda \Omega$ for all $ \lambda \in \cit^*$, we obtain that $(U, \Omega) \simeq (T^*_M, \omega_{can})$ as $\cit^*$-varieties (cf. Lemma 3.7 \cite{F}), where $\omega_{can}$ is the canonical symplectic form on $T^*_M$.
On the other hand,  we have $\dim \pi^{-1}(0) \leq \cfrac{1}{2} \dim W = \dim M$ (Cor. 8.5 \cite{CMSB}). This implies that $M$  is an irreducible component of $\pi^{-1}(0)$. As $M$ is smooth, by Cor. 8.7 \cite{CMSB},
we get that $M \simeq \pit^n$. It follows (for example by the arguments above) that the map $U \to W$ is the Springer map
$\mu: T^*_{\pit^n} \to \bar{\0}_{min} \subset \mathfrak{sl}_{n+1}$.  As $\mu$ is projective, we get $U=Z$. 
\end{proof}

\section{An application}
A {\em symplectic variety} is a normal variety $W$ with a symplectic form
$\omega$ on its smooth locus such that for any resolution $\phi: Z \to W$, 
the 2-form $\phi^* \omega$ defined on $\phi^{-1}(W_{reg})$ extends 
to a regular 2-form on $Z$.

As an application of Theorem \ref{main}, we provide an alternative proof 
of the following main result of \cite{N2}.
\begin{theorem}[Namikawa]
Let $(W, \omega)$ be a singular symplectic variety embedded in $\cit^N$ 
as a complete intersection of hypersurfaces defined by homogeneous polynomials.
Assume that the symplectic form $\omega$ satisfies 
$\lambda^* \omega = \lambda^k \omega, \forall \lambda \in \cit^*$ for some 
$k$. Then $(W, \omega)$ is isomorphic to the nilpotent cone
$(\mathcal{N}, \omega_{KK})$ of a semisimple complex Lie algebra $\mathfrak{g}$ 
together with the Kirillov- Kostant form $\omega_{KK}$.
\end{theorem}

\begin{proof}
By Section 2 of \cite{N2}, $W$ is a conical symplectic variety with 
a symplectic resolution, hence by Theorem \ref{main}, we get that 
$(W, \omega)$ is a nilpotent orbit closure in a semi-simple Lie algebra 
$\fg$. But every nilpotent orbit closure $\bar{\sO} \subset \fg$ 
which is a complete intersection in $\mathfrak{g}$ 
must be the full nilpotent cone, by the main result of Section 7 
of \cite{N2}. 
We now provide a proof of that result, which is somehow shorter and
more uniform than the original one.

Let $d_1, \cdots, d_r \geq 2$ be the degrees of defining equations of 
$\bar{\0}$ in $\fg$. Then $ r = \codim_\fg(\0)$ and by \cite{N2} 
(Section 2, p. 160), we have 
\begin{equation}\label{eq}
\sum_{i=1}^r d_i  = \cfrac{1}{2} \dim \0 + \codim_\fg(\0),  
\quad \text{and} \quad \codim_\fg(\0) \leq \cfrac{1}{3} \dim \fg. 
\end{equation}
We may assume that $\fg$ is simple. We denote by $I$ the ideal of 
$\bar{\sO}$ in the coordinate ring $\cit[\fg]$, and by $\fm$ the maximal 
ideal of $0$ in $\cit[\fg]$. Let $G$ be the adjoint group of $\fg$; then 
$G \times \cit^*$ acts on $\fg$ (where $\cit^*$ acts by dilations) 
and stabilizes $\bar{\sO}$ and $0$. Hence 
$G \times \cit^*$ acts on $\cit[\fg]$ and stabilizes $I \subset \fm$. 
Since $G \times \cit^*$ is reductive, we may find a submodule $M \subset I$ 
which is mapped isomorphically to $I/\fm I$ under the quotient map 
$I \to I/\fm I$. By the graded Nakayama lemma, a homogeneous basis 
of $M$ yields a minimal generating system of the ideal $I$, 
and hence a regular sequence in $\cit[\fg]$ since $\bar{\sO}$ 
is a complete intersection. In geometric terms, the morphism
\[ f : \fg \longrightarrow M^*=: V \] 
corresponding to the inclusion $M \subset \cit[\fg]$ is flat and 
its (scheme-theoretic) fiber $f^{-1}(0)$ equals $\bar{\sO}$ 
(this is a slightly more precise version of Lemma 3 in \cite{N2}). 
Thus, $f$ is open, and hence surjective by $\cit^*$-equivariance. 

Choose a maximal torus $T \subset G$ and denote by $\ft$ its Lie algebra; 
this is a Cartan subalgebra of $\fg$. Since $G \ft$ is dense in $\fg$, 
there exists $x \in \ft$ such that the differential $df_x : \fg \to V$ 
is surjective. As $df_x$ is linear and $T$-equivariant, it follows that 
each weight of the $T$-module $V$ is also a weight of $\fg$. In particular, 
the highest weight of any simple summand of the $G$-module $V$ is either the
highest root, or the highest short root $\lambda$ (if $G$ is not simply laced), 
or $0$. Moreover, the highest root cannot occur, since the corresponding 
simple module is just $\fg$, and 
$\dim(V) = \dim(\fg) - \dim(\bar{\sO}) < \dim(\fg)$.

If $G$ is simply laced, then $V$ must be the trivial $G$-module, 
and hence $\bar{\sO}$ contains the nilpotent cone $\mathcal{N}$; so 
we conclude that $\bar{\sO} = \mathcal{N}$. Thus, we may assume 
that $G$ is not simply laced; then we have an isomorphism of $G$-modules 
\[ V \cong p V(0) \oplus q V(\lambda), \] 
where $p$, $q$ are non-negative integers, $V(0)$ denotes the trivial 
$G$-module $\cit$, and $V(\lambda)$, the simple $G$-module with highest 
weight $\lambda$. (The latter module is called the ``little adjoint module'' 
in \cite{P}, where its invariant theoretical properties are investigated.)
If $q = 0$ then we conclude as above that $\bar{\sO} = \mathcal{N}$; thus, 
we may further assume that $q \geq 1$. 

If $G$ of type $C_n$ (resp. $F_4$, $G_2$), then $V(\lambda)$
has dimension   $2 n^2 - n - 1$ (resp. $26$, $7$),
which contradicts  the inequality 
$\codim_\fg (\0) = \dim V \leq \frac{1}{3} \dim(\fg)$ in \eqref{eq}. 
Thus, we may assume that $G = {\rm SO}(2n + 1)$; then 
$V(\lambda)$ is the natural $G$-module $\cit^{2n+1}$.

Since $f$ is surjective and $G \fg^T$ is dense in $\fg$, it follows
that $G V^T$ is dense in $V$. But this does not hold for 
$V = 2 V(\lambda)$ (as $V(\lambda)^T$ is a line), and hence $q \leq 1$. 
So we may take $V = p V(0) \oplus V(\lambda)$. Since the algebra
of invariant functions $\cit[V(\lambda)]^G$ is generated by the
quadratic form defining $G$, it follows that the categorical quotient 
$V/\!/G := \Spec \, \cit[V]^G$ is an affine space, and the quotient morphism
\[ q_V : V \longrightarrow V/\!/G \]
is flat with reduced fibers. Since $q_V$ sits in a commutative square
\[ \CD
\fg @>{f}>> V \\
@V{q_{\fg}}VV @V{q_V}VV \\
\fg/\!/G @>{f/\!/G}>> V/\!/G, \\
\endCD \]
where $f$, $q_{\fg}$ are flat with reduced fibers, we see that 
$f/\!/G$ is flat with reduced fibers, too.

We claim that $f/\!/G$ is also finite. Consider indeed the restriction
$f^T : \fg^T = \ft \to V^T$. Then the fiber of $f^T$ at $0$ equals
$\bar{\0} \cap \ft = \{ 0 \}$ (as a set). 
Since $f^T$ is equivariant for the natural
actions of $\cit^*$, it follows that $f^T$ is finite. Hence so is
$f^T/W : \fg^T/W \to V^T/W$, where $W$ denotes the Weyl group of
$(G,T)$. But the natural map $\fg^T/W \to \fg/\!/G$ is an isomorphism
by the Chevalley restriction theorem; moreover, the analogous map
$V^T/W \to V/\!/G$ is finite, and dominant since $GV^T$ is dense in $V$. 
The finiteness of  $f/\!/G$ follows from this in view of the commutative square
\[ \CD
\fg^T/W @>{f^T/W}>> V^T/W \\
@VVV @VVV \\
\fg/\!/G @>{f/\!/G}>> V/\!/G. \\
\endCD \]

Since $f/\!/G$ is also flat with reduced fibers, it is a finite \'etale cover, 
and hence an isomorphism as $V/\!/G$ is an affine space. In other words,
the algebra $\cit[\fg]^G$ is freely generated by the pull-backs 
under $f$ of the $p$ projections $V \to V(0)$ and of the basic 
invariant of $V(\lambda)$. The degrees of these invariants are
$a_1,\ldots,a_p,2 a_{p+1}$, where $a_1,\ldots,a_p$ denote the
weights of the action of $\cit^*$ on $p V(0)$, and $a_{p+1}$
the weight of that action on $V(\lambda)$. On the other hand,
the degrees of the basic invariants of $\cit[\fg]^G$ are 
$2,4,\ldots,2n$. In particular, we have $n = p + 1$ and
\[ a_1 + \cdots + a_{n-1} + 2 a_n = n^2 + n. \] 

This implies that $r = \dim V = 3n$ and 
$\dim \0 = \dim \mathfrak{g}- 3n = 2n^2-2n$.
Moreover, we may assume that the degrees of defining equations of $\bar{\0}$
satisfy $d_i = a_i$ for $i = 1, \cdots, n-1$, and $d_n = \cdots = d_{3n} = a_n$.
By \eqref{eq}, we have 
$d_1 + \cdots + d_{3n} = (n^2 - n) + 3 n$,
which gives that 
\[ a_1 + \cdots + a_{n-1} + (2 n + 1) a_n = n^2 + 2 n. \]
Combining the two displayed equalities yields
$(2n - 1) a_n = n$ which is not possible since $a_n = d_n \geq 2$.

\end{proof}

%
%
%
%

\bigskip

Michel Brion 

Institut Fourier, Universit\'e de Grenoble, France 

e-mail: Michel.Brion@ujf-grenoble.fr 

\bigskip
Baohua Fu

Institute of Mathematics, AMSS, 55 ZhongGuanCun East Road,

Beijing, 100190, China

e-mail: bhfu@math.ac.cn


\begin{thebibliography}{KSWZ}
\bibitem[BB]{BB}
Bialynicki-Birula, A.: \emph{Some theorems on actions of algebraic groups},
Ann. of Math., II. Ser. 98 (1973), 480--497
\bibitem[B1]{B1}
Beauville, A.: \emph{Fano contact manifolds and nilpotent orbits}, 
Comment. Math. Helv. 73 (1998), no. 4, 566--583
\bibitem[B2]{B}
Beauville, A.: \emph{Symplectic singularities}, Invent. Math.
139 (2000), 541--549
\bibitem[BR]{BR}
Borel, A. and Remmert, R.:
\emph{\"Uber kompakte homogene K\"ahlersche Mannigfaltigkeiten},
Math. Ann. 145 (1961/1962), 429--439

\bibitem[CMSB]{CMSB}
Cho, K., Miyaoka, Y. and  Shepherd-Barron, N. I.: 
\emph{Characterizations of projective space and applications 
to complex symplectic manifolds}, 
Higher dimensional birational geometry (Kyoto, 1997), 1--88, 
Adv. Stud. Pure Math., 35, Math. Soc. Japan, Tokyo, 2002
\bibitem[F]{F}
Fu, B.: \emph{Symplectic resolutions for nilpotent orbits},
Invent. Math. 151 (2003), 167--186
\bibitem[KPSW]{KPSW}
Kebekus, S., Peternell, T., Sommese, A., Wisniewski, J.: 
\emph{Projective contact manifolds}, Invent. Math. 142 (2000), 1--15. 
\bibitem[N1]{N1}
Namikawa, Y.: \emph{Flops and Poisson deformations of symplectic varieties}, 
Publ. Res. Inst. Math. Sci. 44 (2008), 259--314
\bibitem[N2]{N2}
Namikawa,Y.: \emph{ On the structure of homogeneous symplectic varieties of complete intersection},  
Invent. Math. 193 (2013), no. 1, 159--185 
\bibitem[P]{P}
Panyushev, D.: \emph{Invariant theory of the little adjoint module},
J. Lie Theory 22 (2012), 803--816.  
\end{thebibliography}
\end{document}